\documentclass{amsart}

\usepackage{amsmath}
\usepackage{amssymb}
\usepackage[all]{xy}
\usepackage{picinpar}
\usepackage{palatino}

\def\p{{\mathcal P}_{\infty}}
\def\f{{\mathbb F}_q^{\ast}}
\def\F{{\mathbb F}_q}
\def\pK{\mathfrak p}
\def\fin{\hfill\qed\bigskip}

\def\*#1{#1^*}
\def\lra{\longrightarrow}
\def\A{\Delta\times\prod_{\pK\nmid\infty}U_{\pK}}
\def\B{\prod_{\pK|\infty} K_{\pK}^*\times\prod_{\pK\nmid\infty}U_{\pK}}
\def\simbolo#1#2{\big(#1,#2\big)_{\pK}}
\def\rec#1#2{\frac{\big(#2,k_{\pK}\big(\sqrt[n]{#1}\big)/k_{\pK}\big)
\big(\sqrt[n]{#1}\big)}{\sqrt[n]{#1}}}

\newcommand{\xbinom}{\genfrac(){.5pt}0}
\newcommand{\xxbinom}{\genfrac[]{.5pt}0}

\newcommand{\Gal}{\operatorname{Gal}}
\newcommand{\N}{\operatorname{N}}

\newcounter{bean}

\def\las{\begin{list}
	{{\rm {(\arabic{bean})}}}{\usecounter{bean}
\setlength{\labelwidth}{0.8in}
\setlength{\labelsep}{0.3cm}
\setlength{\leftmargin}{1cm}}}

\numberwithin{equation}{section}
\newtheorem{theorem}{Theorem}[section]
\newtheorem{proposition}[theorem]{Proposition}
\newtheorem{lemma}[theorem]{Lemma}
\newtheorem{remark}[theorem]{Remark}
\newtheorem{definition}[theorem]{Definition}
\newtheorem{corollary}[theorem]{Corollary}

\title[Extended genus field of cyclic Kummer extensions of
rational function fields]
{Extended genus field of cyclic Kummer extensions of
rational function fields}
 
\author[O. Curiel]{Edgar Omar Curiel--Anaya}
\address{Departamento de Control Autom\'atico\\
Centro de Investigaci\'on y de Estudios Avanzados del I.P.N.}
\email{edgaro78@hotmail.com}

\author[M. Maldonado]{Myriam Rosal\'ia Maldonado--Ram\'irez}
\address{Departamento de Matem\'aticas\\
Escuela Superior de F\'isica y Matem\'aticas del Instituto Polit\'ecnico Nacional}
\email{mrmaldonador@ipn.mx}

\author[M. Rzedowski]{Martha Rzedowski--Calder\'on}
\address{Departamento de Control Autom\'atico\\
Centro de Investigaci\'on y de Estudios Avanzados del I.P.N.}
\email{mrzedowski@ctrl.cinvestav.mx}



\date{August 30th., 2021}

\begin{document}

\begin{abstract}

For a cyclic Kummer extension $K$ of a rational function field $k$ is 
considered, via class field theory, the extended Hilbert class field $K_H^+$ 
of $K$ and the corresponding extended genus field $K_g^+$ of $K$ over $k$, 
along the lines of the definitions of R. Clement for such extensions of prime 
degree. We obtain $K_g^+$ explicitly. Also, we use cohomology 
to determine the number of ambiguous classes and obtain a reciprocity 
law for $K/k$. Finally, we present a necessary and sufficient condition for 
a prime of $K$ to decompose fully in $K_g^+$.

\end{abstract}

\maketitle

\section{Introduction}\label{S1}

For a number field $K$, one of the most important arithmetic objets
attached to $K$ is its class group. This group is isomorphic
to the Galois group of the extension $K_H/K$,  where $K_H$ denotes
the maximal  unramified abelian extension of $K$. The field $K_H$
is the {\em Hilbert class field of $K$} (HCF). We have that  $K_H/K$ is a 
finite extension and also that $K_H$ is the abelian extension
of $K$ such that the primes of $K$ that are fully ramified in $K_H$
are precisely the non-zero principal ideals
of $K$. One variant of the HCF is the {\em extended} or
{\em narrow} {\em Hilbert class field of $K$}, denoted by $K_H^+$.
The field $K_H^+$ is the maximal abelian extension of $K$
unramified at the finite primes. We have that $K_H^+/K$ is a finite
extension, that $K_H\subseteq K_H^+$ and also that  $K_H^+$ is 
the abelian extension of $K$ where a prime of $K$ is fully decomposed 
precisely when it is a principal ideal generated by a totally 
positive element, that is, an element such that all its real conjugates are 
positive. 

In order to study the
class group of $K$, but also interesting by itself, it is considered an intermediate field $K\subseteq
K_g \subseteq K_H$, called the {\em genus field of $K$} (relative
to ${\mathbb Q}$). The field $K_g$ is, by definition,  the composite
of $K$ and the maximal
abelian extension of ${\mathbb Q}$ contained in $K_{H}$. That is,
$K_g=Kk^*$, where $\*k$ is the maximal abelian extension of 
${\mathbb Q}$ contained in $K_{H}$. Similarly, it is considered the
{\em extended} or {\em narrow} {\em genus field of $K$} (relative
to ${\mathbb Q}$) $K_g^+$, as the composite of $K$ and the
maximal abelian extension of ${\mathbb Q}$ contained in
$K_H^+$. These definitions are due to A. Fr\"ohlich (\cite{Fro59a,
Fro59b}). For  a number field $K$, the fields $K_H$, $K_H^+$,
$K_g$ and $K_g^+$ are defined without any ambiguity and all of
them are finite extensions of $K$. In particular, when $K/k$ is
an abelian extension, $K_g$ (resp. $K_g^+$) is the maximal
abelian extension of $k$ contained in $K_H$ (resp. $K_H^+$).

When we study global function fields and we want to consider
genus fields and/or extended genus fields, the situation is different
from the number field case since the extensions of constants of
any global function field $K$ are unramified so that the maximal
unramified abelian extension of $K$ is of infinite degree over $K$. 
That is, if we consider the straight analogue of the Hilbert class field as 
the maximal unramified abelian extension of $K$ we have to
deal with infinite extensions. 

There have been a good number
of alternatives to define a Hilbert class field that is a {\em finite} 
extension of a global function field $K$. One of them is to define 
the Hilbert class field of $K$ as the maximal {\em geometric} abelian 
extension of $K$,
that is, the maximal unramified abelian extension of $K$ with the
same field of constants as $K$. It turns out that there are $h_K$
such extensions, where $h_K$ denotes the class number of $K$,
that is, the cardinality of the zero degree divisor class group of $K$ 
which is a finite group. This definition has the issue that $K_H$ is 
not unique but there are $h_K$ different choices. 

To avoid infinite extensions and lack of uniqueness of $K_H$, we have 
to deal with extensions of constants. Since every prime in $K$ is 
eventually inert in an extension of constants, the most accepted way to
define $K_H$ is first to fix a non-empty finite set $S$ of primes
of $K$ and then consider the maximal unramified abelian extension of $K$ 
where the primes of $S$ decompose fully. This such field is denoted by 
$K_{H,S}$ and it is a finite extension of $K$. The Galois group of $K_{H,S}
/K$ is isomorphic to the ideal class group of the Dedekind
ring ${\mathcal O}_S:=\{x\in K\mid v_{\pK}(x)\geq 0\text{\ for
all $\pK\notin S$}\}$. This ideal class group is a finite group.
B. Angl\`es and J.-F. Jaulent \cite{AnJa2000} have given class field theory
definitions of Hilbert class field and extended Hilbert class field
that work for any global field.

R. Clement \cite{Cle92} offered another definition of extended Hilbert class 
field for a cyclic Kummer extension $K$ of $k:={\mathbb F}_q(T)$, the
rational function field, of prime degree $l$ (necessarily  $l|q-1$) and 
consequently another definition of extended genus field $K_g^+$ of $K$ 
(relative to $k$). As far as we know, she was the first one to consider the 
concept of extended genus field for global function fields.

Since the introduction of the concept of {\em genus} by C. F. Gauss,
in the study of quadratic forms and its translation to number fields
by D. Hilbert, 
the concept has been studied by several authors. H. Hasse 
\cite{Has51} was the first to give a definition of genus field by means
of class field theory. Hasse gave his definition for quadratic number
fields. The concept was generalized by H. Leopoldt in \cite{Leo53}
to finite abelian extensions of the field of rational numbers ${
\mathbb Q}$. As a consequence of the work of Hasse, the Galois
group of $K_H^+/K$, where $K$ is a quadratic extension
of ${\mathbb Q}$, is isomorphic to $I_K/P_{K^+}$, where $I_K$
is the group of fractional ideals of $K$ and $P_{K^+}$ is the
subgroup of principal ideals generated by a totally positive element
of $K$. Since $K$ is a quadratic extension of ${\mathbb Q}$, to be
a totally positive element of $K$ is equivalent to have that its norm in
${\mathbb Q}$ is a square of a real number. This concept was 
brought to the case of a cyclic extension $K/k$ of prime degree $l$
with $l|q-1$ by Clement. She defined $K_H^+$ as the class field of
$K$ corresponding to the subgroup $\A$ of the id\`ele group $J_K$
of $K$, where $\Delta:=\big\{(x_{\pK})_{\pK|\infty}\in \prod_{\pK|\infty}
\*{K_{\pK}}\mid \prod_{\N_{K_{\pK}/k_{\infty}}}x_{\pK}\in k_{\infty}^{*l}
\big\}$ and where $\infty$ denotes the infinite prime of $k$, 
that is, the pole of $T$ in $k$ . This definition only works for cyclic
Kummer extensions of $k$ of prime degree.

The aim of this paper is to confirm that the definition of 
$K_H^+$ given by Clement can be extended to general cyclic 
Kummer extensions $K$ of $k$ and  to obtain explicitly the extended 
genus field of a general cyclic Kummer extension of $k$. 
We use cohomology theory to determine the number of ambiguous 
classes. Finally, we obtain a reciprocity law for $K/k$ and present a 
necessary and sufficient condition for 
a prime of $K$ to decompose fully in $K_g^+$. 
We use techniques
similar to the ones used by Clement.

\section{Cyclic Kummer extensions of $k$}\label{S2}

For any global field $E$, $J_E$ denotes the id\`ele group of $E$.
For a place $\pK$ of $E$, $E_{\pK}$ denotes the completion of 
$E$ at $\pK$ and $U_{\pK}$  the group of local units of $E_{\pK}$.
Let $k:=\F(T)$ be the rational function field over the finite field $\F$, $R_T:=
\F[T]$ and $R_T^+:=\{P\in R_T\mid P \text{\ is monic and irreducible}\}$.
The infinite prime $\infty=\p$ of $k$ is the pole of $T$ in $k$. Finally, 
for any $m\in{\mathbb N}$, $C_m$ denotes the cyclic group of order $m$.

Let $n\in{\mathbb N}$ be a natural number dividing $q-1$: $n|q-1$. Let
$K/k$ be a cyclic Kummer extension of degree $n$. Therefore, $K=k\big(
\sqrt[n]{D}\big)$ with $D=\gamma P_1^{\alpha_1}\cdots P_r^{\alpha_r}\in
R_T$, $\gamma\in \f$, $P_1,\ldots,P_r\in R_T^+$ and $1\leq\alpha_i\leq n-1$
for $1\leq i\leq r$. The ramified finite primes are $P_1,\ldots,P_r$. Let
$e_i$ denote the ramification index of $P_i$ in $K/k$, $1\leq i\leq r$. Denote 
by $e_{\infty}$ and $f_{\infty}$ the ramification index
and the inertia degree of any prime $\pK$ in $K$ above $\p$. 

Define
\begin{gather*}
\Delta:=\big\{(x_{\pK})_{\pK|\infty}\mid \prod_{\pK|\infty}
\N_{\*{K_{\pK}}/\*{k_{\infty}}}(x_{\pK})\in k^{*n}_{\infty}\big\}
\subseteq J_K,
\end{gather*}

\begin{gather*}
J_K^+:=\{\vec\alpha\in J_K\mid (\alpha_{\pK})_{\pK|\infty}\in \Delta\}
\intertext{and}
\begin{align*}
K^+:&=\*K\cap J_K^+=\{(x, \ldots,x \ldots)\mid x\in\*K, (x)_{\pK|\infty}\in\Delta\}\\
&=\{x\in\*K\mid \N_{K/k}(x)\in k_{\infty}^{*n}\}.
\end{align*}
\end{gather*}

\begin{lemma}\label{L2.2}
Let $n\in{\mathbb N}$ be a divisor of $q-1$. Then $\frac{k_{\infty}^*}
{k_{\infty}^{*n}}\cong C_n\times C_n$.
\end{lemma}

\begin{proof}
It follows from the group structure of $\*{k_{\infty}}$,
the fact that $n|q-1=|\f|$ and, since $n$ is relatively prime to the
characteristic of $k$, that $\big(U_{\infty}^{(1)}\big)^n=U_{\infty}^{(1)}$, 
where $U_{\infty}^{(1)}$ are the one units of $\*{k_{\infty}}$.
\end{proof}

\begin{lemma}\label{L2.1}
We have 
\las
\item $J_k=\*k\big(\*{k_{\infty}}\times \prod_{P\in R_T^+} U_P\big)$.
\item $\*K J_K^+=J_K$.
\end{list}
\end{lemma}

\begin{proof}
(1) Let $\vec \beta=(\beta_{\infty},\beta_P)_{P\in R_T^+}\in J_k$. Let
$Q_1,\ldots, Q_t\in R_T^+$ be the finite primes such that $v_{Q_i}
(\beta_{Q_i})=c_i\neq 0$. We have that $v_P(\beta_P)=0$ for all
$P\in R_T^+\setminus\{Q_1,\ldots,Q_t\}$. Define $f\in \*k$ as
$f=\prod_{i=1}^t Q_i^{c_i}$. Then $f^{-1}\vec \beta\in
\big(\*{k_{\infty}}\times \prod_{P\in R_T^+} U_P\big)$ and the result
follows.

(2) Let $\vec\alpha\in J_K$. By the approximation theorem, there exists
$x\in\*K$ such that $v_{\pK}(\alpha_{\pK}-x)> v_{\pK}(\alpha_{\pK})$ 
for all $\pK|\infty$. Then $x^{-1}\alpha_{\pK}\in U_K^{(1)}=\big(U_K^{
(1)}\big)^n$ and $\N_{K_{\pK}/k_{\infty}}(x^{-1}\alpha_{\pK})$ $\in k_{\infty}^{
*n}$. Hence $x^{-1}\vec\alpha\in J_K^+$.
\end{proof}

\begin{lemma}\label{L2.3}
The map $\N\colon\frac{\prod_{\pK|\infty}\*{K_{\pK}}}{\Delta}\lra \frac{\*{
k_{\infty}}}{k^{*n}_{\infty}}$ induced by the norm, is injective. Furthermore,
the sequence
\[
1\lra \frac{\prod_{\pK|\infty}\*{K_{\pK}}}{\Delta}\xrightarrow{\phantom{xx}\N
\phantom{xx}}\frac{\*{
k_{\infty}}}{k^{*n}_{\infty}}\xrightarrow{\phantom{xx}\pi
\phantom{xx}} \frac{\*{k_{\infty}}}{\N\big(
\prod_{\pK|\infty}\*{K_{\pK}}\big)}\lra 1,
\]
is exact, where $\N\big(\prod_{\pK|\infty}\*{K_{\pK}}\big)=\{\prod_{\pK|\infty}
\N_{\*{K_{\pK}}/\*{k_{\infty}}}(x_{\pK})\in \*{k_{\infty}}\mid x_{\pK}\in \*{K_{\pK}}\}$.
\end{lemma}

\begin{proof}
Follows from the definition of $\Delta$. 
\end{proof}

\begin{remark}\label{R2.4}{\rm{
For any finite Galois extension $E/F$ of global function fields,
we have that if ${\mathcal P}$ is prime in $F$ and $\pK_1$ and
$\pK_2$ are two primes in $E$ above ${\mathcal P}$, then
$\N_{E_{\pK_1}/F_{\mathcal P}} (\*{E_{\pK_1}})=
\N_{E_{\pK_2}/F_{\mathcal P}} (\*{E_{\pK_2}})$.
}}
\end{remark}

\begin{corollary}\label{C2.5}
We have  $\big[\prod_{\pK|\infty}\*{K_{\pK}}:\Delta\big]=\frac{n^2}{e_{\infty}
f_{\infty}}$.
\end{corollary}

\begin{proof}
From Lemma \ref{L2.3} we obtain that
\[
\Big[\prod_{\pK|\infty}\*{K_{\pK}}:\Delta\Big]=\frac{\big[\*{k_{\infty}}:
k_{\infty}^{*n}\big]}{\big[\*{k_{\infty}}:\N\big(
\prod_{\pK|\infty}\*{K_{\pK}}\big)\big]},
\]
and from Remark \ref{R2.4} we have that $\N\big(
\prod_{\pK|\infty}\*{K_{\pK}}\big)=\N_{\*{K_{\pK}}/\*{k_{\infty}}}
(\*{K_{\pK}})$ for any $\pK|\p$.
From the fundamental result of local field theory, we have that 
$\big[\*{k_{\infty}}:\N_{\*{K_{\pK}}/\*{k_{\infty}}} (\*{K_{\pK}})\big]$
$=
e_{\infty}f_{\infty}$. The result now follows from Lemma \ref{L2.2}.
\end{proof}

\begin{remark}\label{R2.6}{\rm{
We have 
\[
\frac{\prod_{\pK|\infty}\*{K_{\pK}}}{\Delta}\cong \frac{\prod_{\pK|\infty}\*{K_{\pK}}
\times \prod_{P\in R_T^+}U_P}{\Delta\times \prod_{P\in R_T^+}U_P}.
\]
}}
\end{remark}

\begin{lemma}\label{L2.7}
We have the following equalities
\begin{multline*}
\frac{\Big[\prod_{\pK|\infty}\*{K_{\pK}}:\Delta\Big]}{\Big[\*K\cap \Big(
\prod_{\pK|\infty}\*{K_{\pK}} \times \prod_{\pK\nmid\infty}U_{\pK}\Big):
\*K\cap \Big(\Delta \times \prod_{\pK\nmid\infty}U_{\pK}\Big)\Big]}\\
=\frac{\Big[\prod_{\pK|\infty}\*{K_{\pK}}\times \prod_{\pK\nmid\infty}U_{\pK}:
\Delta\times \prod_{\pK\nmid\infty}U_{\pK}\Big]}{\Big[\*K\cap \Big(
\prod_{\pK|\infty}\*{K_{\pK}} \times \prod_{\pK\nmid\infty}U_{\pK}\Big):
\*K\cap \Big(\Delta \times \prod_{\pK\nmid\infty}U_{\pK}\Big)\Big]}\\
=\Big[\*K \Big(\prod_{\pK|\infty}\*{K_{\pK}} \times \prod_{\pK\nmid\infty}U_{\pK}\Big)/\*K:
\*K \Big(\Delta \times \prod_{\pK\nmid\infty}U_{\pK}\Big)/\*K\Big]\\
=\Big[\*K \Big(\prod_{\pK|\infty}\*{K_{\pK}} \times \prod_{\pK\nmid\infty}U_{\pK}\Big):
\*K \Big(\Delta \times \prod_{\pK\nmid\infty}U_{\pK}\Big)\Big].
\end{multline*}
\end{lemma}

\begin{proof}
The first equality follows from Remark \ref{R2.6}. The second equality
is a consequence of the fact that for any finite subgroups $A,B, C$ of an abelian group
$X$ with $A\subseteq B$, we have $\frac{B\cap C}{A\cap C}\cong\frac{CA\cap B}{A}$. 
The last equality is a consequence of the third isomorphism theorem.
\end{proof}

Let ${\mathcal O}_K$ be the integral closure of $R_T$ in $K$. Let $U_K$ be
the group of units of ${\mathcal O}_K$: $U_K=\*{{\mathcal O}_K}$. Set $U_K^+:=
\{\alpha\in U_K\mid \N_{K/k}(\alpha)\in k_{\infty}^{*n}\}=
\{\alpha\in U_K\mid \N_{K/k}(\alpha)\in {\mathbb F}_q^{*n}\}=U_K\cap K^+$.

\begin{lemma}\label{L2.8}
We have
\[
\frac{U_K}{U_K^+}\cong \frac{\*K\cap \Big(
\prod_{\pK|\infty}\*{K_{\pK}} \times \prod_{\pK\nmid\infty}U_{\pK}\Big)}{
\*K\cap \Big(\Delta \times \prod_{\pK\nmid\infty}U_{\pK}\Big)}.
\]
\end{lemma}

\begin{proof}
The natural map 
\begin{eqnarray*}
\varphi\colon U_K&\lra&
\frac{\*K\cap \Big(
\prod_{\pK|\infty}\*{K_{\pK}} \times \prod_{\pK\nmid\infty}U_{\pK}\Big)}{
\*K\cap \Big(\Delta \times \prod_{\pK\nmid\infty}U_{\pK}\Big)}\\
\alpha&\mapsto&(\alpha,\ldots,\alpha,\ldots)\bmod 
\Big(\*K\cap \Big(\Delta \times \prod_{\pK\nmid\infty}U_{\pK}\Big)\Big),
\end{eqnarray*}
is a group epimorphism with $\ker \varphi =U_K^+$.
\end{proof}

\begin{lemma}\label{L2.9}
We have $[U_K:U_K^+]\mid n$.
\end{lemma}

\begin{proof}
Let $\rho\colon U_K\colon \lra \N_{K/k}(U_K)/{\mathbb F}_q^{*n}$ be given
by $\rho(\alpha)=\N_{K/k}(\alpha)\bmod {\mathbb F}_q^{*n}$. Then $\ker\rho=U_K^+$.
It follows that $U_K/U_K^+$ is a subgroup of $\f/{\mathbb F}_q^{*n}\cong C_n$.
\end{proof}

\begin{remark}\label{R2.10}{\rm{
In Lemma \ref{L2.9} we may have $[U_K:U_K^+]<n$. For instance, if $\p$ is
totally inert inert in $K/k$, then $U_K=\f$ and $U_K=U_K^+$.
}}
\end{remark}

\section{Extended Hilbert class field and extended genus field}\label{S3}

Let $K/k$ be a cyclic Kummer extension of degree $n$. We will define
the extended Hilbert class field of $K$ by means of an open subgroup
of finite index in $J_K$. To do this, first we prove the following proposition which
is the generalization of the corresponding one in Clement's paper. We present
the proof for the sake of completeness.

\begin{proposition}\label{P3.1}
The index of $\*K\big(\A\big)$ in the id\`ele group $J_K$ is finite.
\end{proposition}

\begin{proof}
We have that $\*K\big(\A\big)\subseteq \*K\big(\B\big)$. On the one hand we
have that $J_K/\Big(\*K\big(\B\big)\Big)\cong Cl({\mathcal O}_K)$, 
the ideal class group of ${\mathcal O}_K$, which is a finite group. 

On the other hand we have
\begin{multline*}
\Big[\*K\Big(\B\Big):\*K\Big(\A\Big)\Big]\\
=\frac{\Big[\prod_{\pK|\infty}\*{K_{\pK}}:\Delta\Big]}{\Big[\*K\cap\big(\B\big):
\*K\cap\big(\A\big)\Big]}
=\frac{\Big[\prod_{\pK|\infty}\*{K_{\pK}}:\Delta\Big]}{\big[U_K:U_K^+\big]}.
\end{multline*}
The result follows from Corollary \ref{C2.5} and Lemma \ref{L2.9}.
\end{proof}

\begin{remark}\label{R3.2}{\rm{
The group $\Delta$ is the inverse image of $k_{\infty}^{*n}$ under
the norm map, which is a continuous function. Hence the
subgroup $\*K\Big(\A\Big)$ is an open subgroup of $J_K$ of finite
index.
}}
\end{remark}

\begin{definition}\label{D3.3}{\rm{
We define the {\em extended Hilbert class field $K_H^+$} of $K$ as the class
field associated to the id\`ele subgroup $\*K\Big(\A\Big)$ of $J_K$.}}
\end{definition}
\begin{remark}\label{R3.3}{\rm{
We have that $K_H^+/K$ is a finite Galois extension,
\[
\Gal(K_H^+/K)\cong \frac{J_K}{\*K\Big(\A\Big)}
\]
and also that $K_H^+/K$ is unramified at every finite place $\pK$ of $K$.
}}
\end{remark}

\begin{proposition}\label{P3.3'}
We have
\[
\frac{J_K}{\*K\Big(\A\Big)}\cong \frac{J_K^+}{K^+\Big(\A\Big)}\cong
\frac{I_K}{P_K^+},
\]
where $I_K$ is the group of non-zero fractional ideals of ${\mathcal O}_K$,
$P_K$ the subgroup of principal ideals of $I_K$ and $P_K^+$ the subgroup of 
$P_K$ of fractional ideals $(\beta)$ such that $\beta \in K^+$.
\end{proposition}

\begin{proof}
From Lemma \ref{L2.1} we obtain that the natural map 
$\varphi\colon J_K^+\mapsto J_K/\*K$ is surjective and $\ker\varphi=
\*K\cap J_K^+=K^+$. Let $\rho =\hat{\varphi}^{-1}\colon J_K/\*K
\lra J_K^+/K^+$ be the induced isomorphism. Then $\rho\big(\*K
\big(\A)/\*K\big)=\big(J_K^+\cap \*K\big(\A\big)\big)/K^+$. It follows
that
\[
\frac{J_K^+/K^+}{\big(J_K^+\cap \*K\big(\A\big)\big)/K^+}\cong
\frac{J_K/\*K}{\*K\big(\A\big)/\*K}.
\]
The first isomorphism follows since
$J_K^+\cap \*K\big(\A\big)=K^+\big(\A\big)$.

For the second isomorphism consider the map $\theta\colon
J_K^+\lra I_K/P_K^+$ given by $\big((\alpha_{\pK})_{\pK|\infty},(\alpha_{
\pK})_{\pK\nmid \infty}\big)\mapsto \prod_{\pK\nmid\infty}\pK^{v_{\pK}
(\alpha_{\pK})}\bmod P_K^+$. Then $\theta$ is a group epimorphism
and $\ker \theta=K^+\big(\A\big)$.
\end{proof}

\begin{definition}\label{D3.3''}{\rm{
The {\em extended ideal class group of $K$} is defined by
\[
Cl^+\big({\mathcal O}_K\big):=\frac{I_K}{P_K^+}\cong \Gal(
K_H^+/K).
\]
}}
\end{definition}

\begin{proposition}\label{P3.4}
The extension $K_H^+/k$ is a Galois extension.
\end{proposition}

\begin{proof}
It follows from the facts that $\rho(\A)=\A$ for all $k$--embeddings
$\rho$ of $K_H^+$ into a fixed algebraic closure of $K_H^+$ and that
$K/k$ is a Galois extension.
\end{proof}

\begin{proposition}\label{P3.5}
The finite primes in $K$ that decompose fully in $K_H^+$ are precisely the
principal ideals generated by an element $\beta\in \*K$ satisfying $\N_{K/k}
(\beta)\in k_{\infty}^{*n}$.
\end{proposition}

\begin{proof}
From class field theory, see for instance \cite[Corolario 17.6.47]{RzeVil2017},
we have that $\pK$ decomposes fully in $K_H^+/K$ if and only if
$\*{K_{\pK}}\subseteq \*K\big(\A\big)$. Let $\pi$ be such that $v_{\pK}(\pi)=1$. We have
\[
\*{K_{\pK}}\subseteq \*K\Big(\A\Big) \iff (1,1,\ldots,x,1,\ldots )\in \*K\Big(\A\Big)
\]
for each $x\in \*{K_{\pK}}$, in particular for $x=\pi$. Therefore there exist
$\beta\in \*K$ and $\vec\alpha\in \A$ such that $(1,1,\ldots,\pi,1,\ldots)=\beta
\vec\alpha$. It follows that $v_{\mathfrak q}(\beta)=0$ for every finite prime 
${\mathfrak q}\neq \pK$ and $v_{\pK}(\beta^{-1}\pi)=0$. Therefore the only prime
dividing $\langle\beta\rangle$ is $\pK$ and it does so to the power 1. Hence $\pK=
\langle\beta\rangle$.

On the other hand, $(\beta^{-1})_{{\mathfrak q}|\infty}\in \Delta$ so that $\N_{K/k}
(\beta)=\prod_{{\mathfrak q}|\infty}\N_{\*{K_{\mathfrak q}}/\*{k_{\infty}}}(\beta)\in 
k_{\infty}^{*n}$.
\end{proof}

\begin{corollary}\label{C3.6}
If $Q\in R_T^+$ is inert in $K$, then ${\mathfrak q}$ decomposes fully in $K_H^+$
where ${\mathfrak q}=Q{\mathcal O}_K$ is the prime in $K$ above $Q$.
\end{corollary}

\begin{proof}
We have $\N_{K/k}({\mathfrak q})=Q^n$. The result follows.
\end{proof}

\begin{definition}\label{D3.7}{\rm{
We define the {\em extended genus field $K_g^+$ of $K$ (relative to $k$}) as the
maximal abelian extension of $k$ contained in $K_H^+$.
}}
\end{definition}

\begin{remark}\label{R3.8}{\rm{
From class field theory, see for instance \cite[Proposici\'on 17.6.48]{RzeVil2017},
the field $K_g^+$ is the class field associated to $\*k\N_{K/k}\big(\A\big)$.
}}
\end{remark}

\begin{proposition}\label{P3.9}
The degree of $K_g^+$ over $k$ and the degree of $K_g^+$ over $K$ are given by
\[
\big[K_g^+:k\big]=n\prod_{i=1}^r e_i\quad \text{and}\quad \big[K_g^+:K\big]
=\prod_{i=1}^r e_i.
\]
\end{proposition}

\begin{proof}
Let $P\in R_T^+$. Then from Remark \ref{R2.4} we obtain
that $\prod_{\pK|P}\N_{K_{\pK}/k_P}(U_{\pK})=
\N_{K_{\pK}/k_P}(U_{\pK})$ for any fixed prime $\pK|P$. From the
theory of local fields, we have $[U_P:\N_{K_{\pK}/k_P}(U_{\pK})]=
e_P$, the ramification index of $P$ in $K/k$. Recall that 
$e_P=1$ if $P$ is unramified and $e_{P_i}=e_i$, $1\leq i\leq r$.

Therefore, from Lemmas \ref{L2.2} and \ref{L2.1} and since $\*k\cap \big(\*{k_{\infty}}\times
\prod_{P\in R_T^+}U_P\big)=\f$, we obtain
\begin{align*}
\big[K_g^+:k\big]&=\big[J_k/\*k:\big(\*k\N_{K/k}\big(\A\big)\big)/\*k\big]\\
&=\big[\*k\big(\*{k_{\infty}}\times \prod_{P\in R_T^+}U_P\big):
\*k\N_{K/k}\big(\A\big)\Big]\\
&=\frac{\big[\*{k_{\infty}}\times \prod_{P\in R_T^+}U_P:
\N_{K/k}\big(\A\big)\big]}{\big[\*k\cap \big(\*{k_{\infty}}\times 
\prod_{P\in R_T^+}U_P\big):\*k\cap \big(
\N_{K/k}\big(\A\big)\big)\big]}\\
&=\frac{\big[\*{k_{\infty}}\times \prod_{P\in R_T^+}U_P:
\N_{K/k}\big(\A\big)\big]}
{\big[\f:{\mathbb F}_q^{*n}\big]}\\
&=\frac{[\*{k_{\infty}}:k_{\infty}^{*n}]\cdot \prod_{P\in R_T^+}
[U_P:\N_{K_{\pK}/k_P}(U_{\pK})]}{n}=\frac{n^2\prod_{i=1}^re_i}{n}=
n\prod_{i=1}^re_i.
\end{align*}
Finally, since $[K:k]=n$, it follows that $\big[K_g^+:K\big]=\prod_{i=1}^r e_i$.
\end{proof}

Define $\Gamma:={\mathbb F}_{q^n}\big(T,\sqrt[e_1]{P_1},\ldots,
\sqrt[e_r]{P_r}\big)$. Then $[\Gamma:k]=n\prod_{i=1}^r e_i=
[K_g^+:k]$ and $\Gamma/k$ is an abelian extension.
On the other hand, by Abhyankar's Lemma, the ramification
index of $P_i$ in $K\Gamma$ is $e_i$, $1\leq i\leq r$ and $\Gamma/k$
is unramified at every $P\in R_T^+\setminus\{P_1,\ldots,P_r\}$. It follows
that $\Gamma/K$ is unramified at every finite prime $P\in R_T^+$.

We are ready to prove our main result, which gives an explicit and nice 
expression for $K_g^+$.  

\begin{theorem}\label{T3.10}
Let $n\in{\mathbb N}$ be a natural number dividing $q-1$: $n|q-1$. Let
$K/k$ be a cyclic Kummer extension of degree $n$, $K=k\big(
\sqrt[n]{D}\big)$ with $D=\gamma P_1^{\alpha_1}\cdots P_r^{\alpha_r}\in
R_T$, $\gamma\in \f$, $P_1,\ldots,P_r\in R_T^+$ and $1\leq\alpha_i\leq n-1$
for $1\leq i\leq r$. The ramified finite primes are $P_1,\ldots,P_r$. Let
$e_i$ be the ramification index of $P_i$ in $K/k$, $1\leq i\leq r$.

Then
\[
K_g^+=\Gamma={\mathbb F}_{q^n}\big(T,\sqrt[e_1]{P_1},\ldots,
\sqrt[e_r]{P_r}\big).
\]
\end{theorem}

\begin{proof}
It suffices to prove that $\Gamma\subseteq K_H^+$ since $K_g^+$
is the maximal abelian extension of $k$ contained in $K_H^+$
and $\Gamma/k$ is an abelian extension. Now, let 
$H:=\Gal(\Gamma/k)\cong C_n\times C_{e_1}\times\cdots
\times C_{e_r}$. Since $e_i|n$ for all
$1\leq i\leq r$,  $H$ is of exponent $n$. Therefore, it is enough to show
that any abelian extension of $k$, containing $K$, of exponent $n$ and
such that it is unramified at the finite primes of $K$, is contained in
$K_H^+$.

Let $L$ be such an extension. By class field theory, it is enough to prove
that $\*K\big(\A\big)\subseteq \*K\N_{L/K}(J_L)$. We have the following
commutative diagram
\[
\xymatrix{
J_K\ar@{->}[rr]^{\rho_K}\ar@{->}[d]_{\N_{K/k}}&&\Gal(L/K)
\ar@{->}[d]^{\iota}\\ 
J_k\ar@{->}[rr]_{\rho_k}&&\qquad{\ }\qquad{\ }&
\hspace{-73pt}\Gal(L/k)\cong C_{m_1}\times\cdots\times C_{m_t},
}
\]
where $\rho_K$ and $\rho_k$ denote Artin's reciprocity maps,
$\iota$ is the natural embedding and $m_j|n$, $1\leq j\leq t$.
The norm of an element $\vec\alpha\in \Delta$ is of the form $(\beta,
1,\ldots,1, \ldots )\in J_k^n$. Therefore $(\beta,1,\ldots,1, \ldots )\in
\ker \rho_k$. Hence $\rho_K(\Delta)\in \ker \rho_K=\*K\N_{L/K}(J_L)$.
Since $L/K$ is unramified at every finite prime, it follows that $U_{\pK}
\subseteq \*K\N_{L/K}(J_L)$ for every finite prime $\pK$. Therefore
$\A\subseteq \*K\N_{L/K}(J_L)$. The result follows.
\end{proof}

\section{ambiguous classes}\label{S4}

We understand by {\em ambiguous classes} the elements of $Cl^+({\mathcal
O}_K)$ fixed under the action of $G:=\Gal(K/k)$:
$Cl^+({\mathcal O}_K)^G$. We are interested in the number of such classes.

Let $G=\Gal(K/k)=\langle\sigma\rangle$. Let $\rho\colon Cl^+({\mathcal O}_K)
\lra Cl^+({\mathcal O}_K)^{1-\sigma}$ be the map $[{\mathfrak a}]\mapsto
[{\mathfrak a}][{\mathfrak a}]^{-\sigma}$ for ${\mathfrak a}\in I_K$ and
$[{\mathfrak a}]={\mathfrak a}\bmod P_K^+$. Then $\rho$ is an epimorphism and
$\ker\rho=Cl^+({\mathcal O}_K)^G$. In particular, $\frac{Cl^+({\mathcal O}_K)}
{Cl^+({\mathcal O}_K)^{1-\sigma}}\cong Cl^+({\mathcal O}_K)^G$.
Let ${\mathcal G}:=\Gal(K_H^+/k)$. Since $K_g^+$ is the maximal abelian
extension of $k$ contained in $K_H^+$, we have that the commutator
subgroup ${\mathcal G}'$ is isomorphic to $\Gal(K_H^+/K_g^+)$.
\[
\xymatrix{
&&K_g^+\ar[r]^{{\mathcal G}'}&K_H^+\\
K\ar@/^1pc/[rru]|{Cl^+({\mathcal O}_K)/{\mathcal G}'}\ar[rrru]|<(.2){{Cl^+
(\mathcal O}_K)}\\
k\ar[u]^G\ar@/_2pc/[rrruu]_{\mathcal G}\ar@/_1pc/[rruu]|{{\mathcal G}/{\mathcal G}'}
}
\]

Now, we have that ${\mathcal G}'\cong Cl^+({\mathcal O}_K)^{1-\sigma}$.
To find $|Cl^+({\mathcal O}_K)^G|$ we need several results
on cohomology theory, most of them well known.

First, we have the exact sequence
\[
1\lra K^+\lra \*K\lra \*K/K^+\lra 1,
\]
From Hilbert's theorem 90, we have $H^1(G,\*K)=\{1\}$, therefore we obtain
the cohomology exact sequence
\[
1\lra \*k\lra \*k\lra (\*K/K^+)^G\lra H^1(G,K^+)\lra 1,
\]
so that $H^1(G,K^+)\cong (\*K/K^+)^G$. We have, for any $a\in \*K$,
$\sigma(a)/a\in K^+$, which implies that $(\*K/K^+)^G=\*K/K^+$.
Using the approximation theorem, we obtain that
\begin{gather*}
\*K/K^+\cong \big(\prod_{\pK|\infty}\*{K_{\pK}}\big)/\Delta.
\intertext{From Corollary \ref{C2.5} it follows that}
|H^1(G,K^+)|=\frac{n^2}{e_{\infty}f_{\infty}}.
\end{gather*}

\begin{lemma}\label{L4.1}
The Herbrand quotient of $U_K$ is $h(G,U_K)=\frac{e_{\infty}f_{\infty}}{n}$.
\end{lemma}

\begin{proof}
From Dirichlet's unit theorem, we have that $U_K\cong{\mathbb Z}^{m-1}\times \f$
where $m$ is the number of primes of $K$ above the infinite prime $\p$ of $k$. 
Let $\pK_1, \ldots, \pK_m$ be the primes of $K$ that lie above $\p$, ordered such that
$\sigma(\pK_j)=\pK_{j+1}$ for $1\leq j\leq m-1$ and $\sigma(\pK_m)=\pK_1$.
Choose $a\in{\mathbb N}$ such that $\big(\frac{\pK_j}{\pK_{j+1}}\big)^a=
\langle\mu_j\rangle$ is a principal ideal and $\mu_j\in U_K$, for all $1\leq j
\leq m-1$. We have $\sigma(\mu_j)=\mu_{j+1}$ for $1\leq j\leq m-2$ and
$\sigma(\mu_{m-1})=(\mu_1\cdots \mu_{m-1})^{-1}=:\mu_m$. Thus $\sigma
(\mu_m)=\mu_1$. 

It follows that $V:=\langle \mu_1,\ldots,\mu_{m-1}\rangle$ is a $G$--submodule
of $U_K$ of finite index. Furthermore, $V\cong \big({\mathbb Z}[G/D]\big)/
{\mathbb Z}$ as $G$--modules, where $D$ is the decomposition group of
any of the primes of $K$ above $\p$.

We have an exact sequence of $G$--modules
\begin{gather*}
1\lra V\lra U_K\lra F\lra 1,
\intertext{where $F$ is finite. Then we have $h(G,U_K)=h(G,V)$. Now,
from the exact sequence of $G$--modules}
1\lra {\mathbb Z}\lra {\mathbb Z}[G/D]\lra V\lra 1,
\intertext{we obtain that}
h(G,V)=\frac{h(G,{\mathbb Z}[G/D])}{h(G,{\mathbb Z})}=\frac{n/m}{n}=\frac 1m
=\frac{e_{\infty}f_{\infty}}{n}.
\end{gather*}
\end{proof}

\begin{lemma}\label{L4.2}
We have $|H^1(G,U_K^+)|=n^2/e_{\infty}f_{\infty}$.
\end{lemma}

\begin{proof}
Since $U_K/U_K^+$ is finite, it follows that $h(G,U_K)=h(G,U_K^+)=e_{\infty}
f_{\infty}/n$. Now, we have the Tate cohomology group
\[
\hat{H}^0(G,U_K^+)=\frac{(U_K^+)^G}{\N_{K/k}(U_K^+)}=\frac{\f}{
{\mathbb F}_q^{*n}}\cong C_n.
\]
The result follows.
\end{proof}

\begin{lemma}\label{L4.3}
We have $|I_K/I_k|=e_1\cdots e_r$.
\end{lemma}

\begin{proof}
For any $P\in R_T^+$, let ${\mathfrak a}_P=\big(\prod_{\pK|P}\pK)^{e_P}$ be the
conorm of $P$, where $e_P$ denotes the ramification index of $P$ in $K/k$.
Then $I_K^G$ is the free abelian group with free generators 
$\big\{{\mathfrak a}_P\big\}_{P\in R_T^+}$. Since $I_k$ is the free abelian group with generators
$\{P\}_{P\in R_T^+}=\big\{{\mathfrak a}_P^{e_P}\big\}_{P\in R_T^+}$ and the ramified finite primes
are $P_1,\ldots,P_r$ with ramification indices $e_1,\ldots, e_r$, we get the result.
\end{proof}

\begin{theorem}\label{T4.4}
The number of ambiguous classes $\big|Cl^+({\mathcal O}_K)^G\big|$ is equal to
$e_1\cdots e_r$.
\end{theorem}

\begin{proof}
From the exact sequence $1\lra P_K^+\lra I_K\lra Cl^+({\mathcal O}_K)\lra 1$, and
since $H^1(G,I_K)=\{1\}$, we obtain the cohomology sequence
\begin{gather*}
1\lra (P_K^+)^G\lra I_K^G\lra Cl^+({\mathcal O}_K)^G\lra H^1(G,P_K^+)\lra 1.
\intertext{Dividing the first two terms by $I_k=P_k\subseteq P_K^G$, we obtain}
\big|Cl^+({\mathcal O}_K)^G\big|=\frac{\big|I_K^G/I_k\big|}{\big|(P_K^+)^G/I_k\big|}
\cdot \big|H^1(G, P_K^+)\big|.
\end{gather*}

Next, we consider the exact sequence of $G$--modules
\begin{gather*}
1\lra U_K^+\lra K^+\lra P_K^+\lra 1.
\intertext{Since $\big(U_K^+\big)=\f$ and $(K^+)^G=\*k$, we obtain the exact 
cohomology sequence}
1\lra \f\lra \*k\lra \big(P_K^+\big)^G\lra H^1(G,U_K^+)\lra H^1(G,K^+)\\
\lra H^1(G,P_K^+)\xrightarrow{\ \nu\ } H^2(G,U_K^+)\xrightarrow{\ \rho\ }H^2(G,K^+)\lra \cdots
\end{gather*}

Now, we have that $I_k\cong\*k/\f$, that
\begin{gather*}
H^2(G,U_K^+)\cong H^0(G,U_K^+)=\frac{\big(U_K^+\big)^G}{\N_{K/k}\big(U_K^+\big)}=
\frac{\f}{{\mathbb F}_q^{*n}},\\
H^2(G,K^+)\cong H^0(G,K^+)=\frac{\big(K^+\big)^G}{\N_{K/k}\big(K^+\big)}=
\frac{\*k}{\N_{K/k}(K^+)}
\intertext{and that $\rho$ is an injective map. Therefore, we obtain the exact sequence}
1\lra \frac{\big(P_K^+\big)^G}{I_k}\lra H^1(G, U_K^+)\lra H^1(G, K^+)\lra H^1(G, P_K^+)
\lra 1.
\intertext{Therefore}
\frac{\big|H^1(G,P_K^+)\big|}{\big|\big(P_K^+\big)^G/I_k\big|}=
\frac{\big|H^1(G,K^+)\big|}{\big|H^1(G,U_K^+)\big|}.
\end{gather*}
The result now follows from Lemma \ref{L4.3}.
\end{proof}

\begin{theorem}\label{T4.5}
We have 
\begin{gather*}
\Gal(K_g^+/K)\cong\frac{Cl^+({\mathcal O}_K)}{Cl^+({\mathcal O}_K)^{
1-\sigma}}\cong Cl^+({\mathcal O}_K)^G. 
\end{gather*}
\end{theorem}

\begin{proof}
From the isomorphism $\frac{Cl^+({\mathcal O}_K)}
{Cl^+({\mathcal O}_K)^{1-\sigma}}\cong Cl^+({\mathcal O}_K)^G$
 and Theorem \ref{T4.4}, we obtain that $\big|
Cl^+({\mathcal O}_K)^{1-\sigma}\big|=\big|{\mathcal G}'\big|=
\big[K_H^+:K_g^+\big]$, where
${\mathcal G}=\Gal(K_H^+/k)$.

Let $\rho\colon Cl^+({\mathcal O}_K)\lra \Gal(K_H^+/K)
\subseteq {\mathcal G}$ be the Artin reciprocity map. For any
${\mathfrak b}\in Cl^+({\mathcal O}_K)$ we have
\begin{align*}
\rho\big({\mathfrak b}^{1-\sigma}\big)&=\rho({\mathfrak b})\rho\big({\mathfrak b}^{-\sigma}\big)
=\rho({\mathfrak b})\rho\big({\mathfrak b}^{\sigma}\big)^{-1}=
\rho({\mathfrak b})\Big(\sigma^{-1}\rho({\mathfrak b})\sigma\Big)^{-1}\\
&=\rho({\mathfrak b})\sigma^{-1} \rho({\mathfrak b})^{-1}\sigma\in{\mathcal G}'.
\end{align*}
Hence $\rho\big(Cl^+({\mathcal O}_K)^{1-\sigma}\big)=
{\mathcal G}'$ and we get the result. 
\end{proof}

\section{A reciprocity law for $K/k$}\label{S5}

Here we present a reciprocity law which is analogous to the quadratic
reciprocity law. Let $K=k\big(\sqrt[n]{D}\big)$ be 
as in Section \ref{S2}. Let $Q\in R_T^+$ be such that $Q\nmid D$.
Let ${\mathfrak q}$ be a prime in $K$ above $Q$. The extension $K_{\mathfrak q}/
k_Q$ of local fields is unramified of degree $f$, the inertia degree of 
${\mathfrak q}/Q$. We denote the residue fields by $\hat K$ and $\hat k$ 
respectively. If $Q$ is of degree $d$, then $|\hat k|=q^d$ and $|\hat K|
=q^{df}$. We denote by $\varphi_Q$ the element of $\Gal(K/k)$ that 
corresponds to the Frobenius generator of
$\Gal(\hat K/\hat k)$. Then $\varphi_Q$ is given by
\begin{gather*}
\varphi_Q\big(\sqrt[n]{D}\big)\equiv\big(\sqrt[n]{D}\big)^{q^d}\bmod \mathfrak q,\\
\intertext{that is,}
\frac{\varphi_Q\big(\sqrt[n]{D}\big)}{\sqrt[n]{D}}\equiv
D^{\frac{q^d-1}{n}}\bmod {\mathfrak q}.
\end{gather*}
Since $n|q^d-1$, both sides of the congruence belong to $k$.
Furthermore there exists $j\in{\mathbb N}$ such that $
\varphi_Q\big(\sqrt[n]{D}\big)/\sqrt[n]{D}=\zeta_n^j$, where $\zeta_n$
is a primitive $n$--th root of unity.

\begin{definition}{\rm{
We define the {\em residue symbol}
\begin{gather*}
\xbinom DQ_n\in \f
\intertext{as the unique $n$--th root of unity satisfying}
\xbinom DQ_n\equiv D^{\frac{q^d-1}{n}}\bmod Q.
\intertext{More generally, if $R=\prod_{j=1}^t Q_j^{\alpha_j}\in R_T$ is relatively prime to $D$,}
\xbinom DR_n:=\prod_{j=1}^t\xbinom D{Q_j}_n^{\alpha_j}.
\intertext{Equivalently, if ${\mathfrak a}$ is a non-zero ideal of $R_T$ relatively prime to $D$,}
\xbinom D{\mathfrak a}_n:=\prod_{P\in R_T^+}\xbinom DP_n^{v_P({\mathfrak a})}.
\end{gather*}
}}
\end{definition}

Note that $Q$ decomposes fully in $K$ if and only if $\xbinom DQ_n=1$.

The main properties of the symbol $\xbinom DQ_n$ are given in the
following proposition, we omit the straightforward proof.

\begin{proposition}\label{P5.2}
We have
\las
\item Let $C,D\in R_T$ and $Q\in R_T^+$ be such that $Q\nmid CD$. Then
\[
\xbinom CQ_n\xbinom DQ_n=\xbinom {CD}Q_n.
\]
\item For $Q\nmid D$, we have $\xbinom DQ_n=1$ if and only if
$D\bmod Q\in \big((R_T/\langle Q\rangle)^{*}\big)^n$.
\item For $a\in\f$, 
\begin{gather*}
\xbinom aQ_n=a^{\frac{q^d-1}{n}}.
\end{gather*}
\end{list}
\hfill $\fin$
\end{proposition}

\begin{definition}\label{D5.3}{\rm{
Let $\pK$ be a prime in $k$ and let $R,S\in R_T$ be
two relatively prime non-zero polynomials: $\gcd(R,S)=1$.
We define the {\em Hilbert norm residue symbol} by
\[
\simbolo RS:=\rec RS,
\]
}} 
{\rm where} $\big(S,k_{\pK}\big(\sqrt[n]{R}\big)/k_{\pK}\big)$ {\rm denotes the local norm residue symbol.}

\end{definition}

We have the following {\em symbol product formula}.  
\[
\prod_{\pK}\simbolo RS=1,
\]
where $\pK$ runs through all the prime divisors of $k$, from which 
it is obtained the following {\em reciprocity law}. 
\begin{theorem}\label{T5.4}
Let $Q,R\in R_T^+$ be of degrees $\delta(Q)$ and $\delta(R)$ 
respectively. Then
\[
\xbinom Q{\langle R\rangle}_n \cdot 
\xbinom R{\langle Q\rangle}_n^{-1}=
\xxbinom {(-1)^{\delta(Q)\delta(R)}b_0^{\delta(Q)}}
{a_0^{\delta(R)}}^{\frac{q-1}{n}}=1.
\]
\end{theorem}

\begin{proof}
Similar to \cite[Proposition 4.1]{Cle92}.
\end{proof}

Finally, we give our generalization to Theorem 4.2
\cite{Cle92}.

\begin{theorem}\label{T5.5}
We have that a prime $\pK$ of ${\mathcal O}_K$ decomposes
fully in $K_g^+$ if and only if each finite prime of $k$ ramified in $K$,
that is, each $P_j$, $1\leq j\leq r$, decomposes fully in $k\big(
\sqrt[n]{B})/k$, where $B$ is a monic generator of $\N_{K/k}\pK$ 
and  $n$ divides $\deg B$.
\end{theorem}

\begin{proof}
Let $d_j:=\deg P_j$ and $P_j^*:=(-1)^{d_j}P_j$, $1\leq j\leq r$.
We have that $\pK$ decomposes fully in $K_g^+/K$
if and only if the Artin symbol $(\pK,K_g^+/K)=1$. Since $K_g^+={\mathbb F}_{q^n}\big(\sqrt[e_1]{P_1},
\ldots, \sqrt[e_r]{P_r}\big)={\mathbb F}_{q^n}\big(\sqrt[e_1]{P_1^*},
\ldots, \sqrt[e_r]{P_r^*}\big)$, we have $(\pK,K_g^+/K)=1$ if and
only if $(\pK,K_g^+/K)|_{k\big(\sqrt[e_j]{P_j^*}\big)}=1
$ for all $1\leq j\leq r$, and $(\pK,K_g^+/K)|_{
{\mathbb F}_{q^n}(T)}=1$. This is equivalent to 
\begin{gather*}
(\N_{K/k}\pK,k\big(\sqrt[e_j]{P_j^*}\big)/k)=1 \iff 
\xbinom {P_j^*}{\N_{K/k}\pK}_{e_j}=1\quad\text{for all}\quad 1\leq j
\leq r,
\intertext{and}
(\N_{K/k}\pK,{\mathbb F}_{q^n}(T)/k)=1 \iff
\xbinom {\xi}{\N_{K/k}\pK}_n=1,
\end{gather*}
where $\xi$ is a generator of $\f$.

Let $h=\deg B$. Then, by the reciprocity law, 
\begin{align*}
\xbinom {P_j^*}{\N_{K/k}\pK}_{e_j}&=\xbinom {-1}{\N_{K/k}\pK}_{e_j}^{d_j}
\xbinom {P_j}{\N_{K/k}\pK}_{e_j}\\
&=(-1)^{((q^h-1)/e_j)d_j}(-1)^{hd_j(q-1)/e_j}
\xbinom {B}{\langle P_j\rangle}_{e_j}=\xbinom {B}{\langle P_j\rangle}_{e_j},
\end{align*}
for $1\leq j\leq r$. 

Therefore, $\pK$ decomposes fully in $K_g^+/K$ if and only
if $\big(P_j(T),k\big(\sqrt[n]{B}\big)/k\big)=1$ for $1\leq j\leq r$ and 
$\xi^{(q^h-1)/n}=1$. The last equality is equivalent to $n|h$
since the order of $\xi$ in $\f$ is $q-1$ and $q\equiv 1\bmod n$.
\end{proof}

\end{document}